\documentclass[11pt]{article}
\usepackage{amsthm}

\usepackage{amsmath,amssymb,amsfonts,amsxtra}
\usepackage{verbatim}
\usepackage{times}
\usepackage{enumerate}
\usepackage{url}

\newtheorem{thm}{Theorem}

\newtheorem{lemma}[thm]{Lemma}

\numberwithin{equation}{section}

\newcommand{\Comment}[1]{}
\newcommand{\rbr}[1]{\left( {#1} \right)}

\newcommand{\cbr}[1]{\left\{ {#1} \right\}}

\newcommand{\abs}[1]{\left| {#1} \right|}

\newcommand{\eq}[1]{(\ref{#1})}

\usepackage{environ}
\NewEnviron{as}{\begin{align}\begin{split}\BODY\end{split}\end{align}}
\NewEnviron{as*}{\begin{align*}\begin{split}\BODY\end{split}\end{align*}}

\def\one{\mathbf{1}}

\def\CC{\mathbb{C}}
\def\RR{\mathbb{R}}

\def\NN{\mathbb{N}}

\def\supp{\text{supp}}

\makeatletter
\newcommand{\classno}[2][2010]{%
  \let\@oldtitle\@title%
  \gdef\@title{\@oldtitle\footnotetext{#1 \emph{Mathematics Subject Classification} #2}}%
}
\newcommand{\extraline}[1]{%
  \let\@@oldtitle\@title%
  \gdef\@title{\@@oldtitle\footnotetext{ #1}}%
}
\makeatother

\begin{document}

\title{Sharpness of the Mockenhaupt-Mitsis-Bak-Seeger Restriction Theorem in Higher Dimensions}

\author{Kyle Hambrook and Izabella {\L}aba}

\classno{42A38, 42B10, 42B15, 42B20, 28A80}


\extraline{The authors were supported by the NSERC Discovery Grant 22R80520 and an NSERC Postdoctoral
Fellowship.}


\maketitle

\begin{abstract}
We prove the range of exponents in the general $L^2$ Fourier restriction theorem due to Mockenhaupt, Mitsis, Bak and Seeger is sharp for a large class of measures on $\mathbb{R}^d$. This extends to higher dimensions the sharpness result of Hambrook and {\L}aba.
\end{abstract}

\section{Introduction}

If $f: \RR^d \rightarrow \CC$ is Lebesgue integrable, then the Fourier transform of $f$ is
$$
\widehat{f}(\xi) = \mathcal{F}[f](\xi) = \int e^{-2\pi i x \cdot \xi} f(x)  dx \qquad \forall \xi \in \RR^d.
$$
If $\mu$ is a measure on $\RR^d$ and $f: \RR^d \rightarrow \CC$ is $\mu$-integrable, then the Fourier-Stieltjes transform of the measure $f \mu$ is
$$
\widehat{f\mu}(\xi) = \mathcal{F}[f\mu](\xi) = \int e^{-2\pi i x \cdot \xi} f(x)  d\mu(x) \qquad \forall \xi \in \RR^d.
$$
The expression $X \lesssim Y$ stands for ``there exists a constant $C > 0$ such that $X \leq C Y$.'' The expression $X \gtrsim Y$ is analogous, and $X \approx Y$ means that $X \lesssim Y$ and $X \gtrsim Y$. 

The purpose of this paper is to address the sharpness of the range of exponents in the following general $L^2$ restriction theorem.

\begin{thm}\label{mock-mit theorem}
Suppose that $\mu$ is a finite Borel measure on $\RR^d$. Suppose that there are $\alpha, \beta \in (0,d)$ such that 
\begin{align}
\label{A}
\mu(B(x,r)) &\lesssim r^{\alpha} \quad \forall x \in \RR^d, r > 0, \\
\label{B}
|\widehat{\mu}(\xi)| &\lesssim (1+|\xi|)^{-\beta/2} \quad \forall \xi \in \RR^d.
\end{align}
Then for all $p \geq (4d-4\alpha+2\beta)/\beta$ we have 
\begin{align}\label{res-ineq}
\| \widehat{f d\mu} \|_{p} \lesssim  \| f \|_{L^2(\mu)} \qquad \forall f \in L^2(\mu).
\end{align}
\end{thm}

Theorem \ref{mock-mit theorem} was proved independently by Mockenhaupt \cite{mock} and Mitsis \cite{mit} in the non-endpoint range
$p > (4d-4\alpha+2\beta)/\beta$; the endpoint case was established later by Bak and Seeger \cite{BS}.

Theorem \ref{mock-mit theorem} generalizes the classical Stein-Tomas restriction theorem for the sphere and its variants for other smooth submanifolds of $\RR^d$ for $d \geq 2$ (cf. \cite{stein-ha} for a discussion of such theorems). For example, the Stein-Tomas restriction theorem for the sphere says that \eq{res-ineq} holds for $p \geq (2d+2)/(d-1)$ when $\mu$ is a uniform measure on a sphere in $\RR^d$. Theorem \ref{mock-mit theorem} covers this case, with $\alpha = \beta = d-1$.
However, Theorem \ref{mock-mit theorem} is not limited to the classical setting of surface measures on smooth manifolds. It also applies (for instance) to measures on fractals in $\RR^d$, including when $d=1$. 



The range $p \geq (2d+2)/(d-1)$ in the classical Stein-Tomas restriction theorem for the sphere is sharp. This is illustrated by the so-called Knapp example, where $f$ is taken to be the indicator function of a small (hence almost flat) spherical cap. Similar examples can be constructed for other smooth manifolds.

For fractal measures, the question of sharpness of the range of $p$ in Theorem \ref{mock-mit theorem} remained open until recently. In this setting, it turns out that the availability of restriction estimates beyond that range depends on the particulars of the measure under consideration.

On one hand, the range of $p$ in Theorem \ref{mock-mit theorem} is sharp for the class of measures on $\RR$ that satisfy \eq{A} and \eq{B} with $0 < \beta \leq \alpha < 1$.
This is due to Hambrook and {\L}aba \cite{HL}  (cf. \cite{hambrook-thesis}), who proved a variant of Theorem \ref{HL thm} below 
where 
the measure $\mu$ depends on $p$. The main theorem of \cite{HL} addresses the case of Salem measures, where $\beta$ and $\alpha$ can be taken arbitrarily close together,
but it is easy to adapt the same argument to allow more general $0<\beta<\alpha<1$, see \cite{hambrook-thesis}. Chen \cite{chen} modified the argument of \cite{HL} to obtain the version stated here.
\begin{thm}\label{HL thm}
Given any $0 < \beta \leq \alpha < 1$, there is a compactly supported probability measure $\mu$ on $\RR$ that satisfies \eq{A} and \eq{B} but does not satisfy \eq{res-ineq} for any $p < (4-4\alpha+2\beta)/\beta$. In particular, there is a sequence of functions $f_l \in L^2(\mu)$ such that 
\begin{align*}
\lim_{l \rightarrow \infty} \dfrac{ \| \widehat{f_l d\mu} \|_{p} }{ \| f_l \|_{L^2(\mu)} } = \infty \quad \forall p < (4-4\alpha+2\beta)/\beta.
\end{align*}
\end{thm}

On the other hand, there exist specific fractal measures on $\RR^d$ obeying \eq{A} and \eq{B} for which the estimate (\ref{res-ineq}) holds for 
a better range of $p$ than that provided in Theorem 1. 
Such measures were constructed by Shmerkin and Suomala \cite{shmerkin-suomala} for $d=1$ and $\alpha>1/2$ (the same proof also works for $d=2,3$ and $d/2<\alpha\leq 2$), and, independently via a different method, by Chen and Seeger \cite{chen-seeger2015} for all $d\geq 1 $ and $\alpha=\beta=d/k$, where $k\in\NN$. The best possible range for a measure supported on a set of Hausdorff dimension $\alpha$ 
is $p\geq 2d/\alpha$ (this follows easily from energy estimates, see, for example, \cite{HL}), and that range is in fact achieved in \cite{chen-seeger2015}.
An earlier paper by Chen \cite{chen2} provides an example of a measure on $\RR$ supported on a set of Hausdorff dimension $1/2$ for which  (\ref{res-ineq}) holds for all $p\geq 4$, but that measure does not obey \eq{A} or \eq{B}.

Theorem \ref{HL thm} says that the range of $p$ in Theorem \ref{mock-mit theorem} is sharp for the class of measures on $\RR$ such that \eq{A} and \eq{B} hold with $0 < \beta \leq \alpha < 1$, but it says nothing about measures on $\RR^d$ for $d \geq 2$. The construction of \cite{HL} (or \cite{chen}) does not appear to 
generalize in a straightforward manner to higher dimensions. However, we are able to combine it with the classical Knapp example to prove the
following sharpness theorem, which is the main result of this paper.

%
%

\begin{thm}\label{main-thm} Let $d \geq 2$. 
Given any $d-1 < \beta \leq \alpha < d$, there is a compactly supported probability measure $\nu$ on $\RR^d$ that satisfies \eq{A} and \eq{B} but does not satisfy \eq{res-ineq} for any $p < (4-4\alpha+2\beta)/\beta$. In particular, there is a sequence of functions $f_l \in L^2(\nu)$ such that 
\begin{align*}
\lim_{l \rightarrow \infty} \dfrac{ \| \widehat{f_l d\nu} \|_{p} }{ \| f_l \|_{L^2(\nu)} } = \infty \quad \forall p < (4d-4\alpha+2\beta)/\beta.
\end{align*}
\end{thm}

Theorem \ref{main-thm} says that when $d \geq 2$ the range of $p$ in Theorem \ref{mock-mit theorem} is sharp for the class of measures on $\RR^d$ such that \eq{A} and \eq{B} hold with $d-1 < \beta \leq \alpha < d$. Readers interested in future research should note that the problem remains open when $d \geq 2$ and $\alpha$, $\beta$ do not satisfy $d-1 < \beta \leq \alpha$.

The proof of Theorem \ref{main-thm} adapts the Hambrook-{\L}aba construction in dimension 1 (with the modifications due to Chen \cite{chen}) to higher dimensions by combining it with the classical Knapp example.


The starting point for the proof of Theorem \ref{HL thm} in \cite{HL}, \cite{chen} is a  construction due to {\L}aba and Pramanik \cite{LP} of random Cantor sets in $\RR$ of dimension $0 < \alpha < 1$ whose natural measures $\mu$ satisfy conditions \eq{A} and \eq{B} for any $\beta < \alpha$. 
The key new idea of \cite{HL} was that such sets could be modified to include lower-dimensional deterministic Cantor-type subsets that have far more arithmetic structure than the rest of the set. As long as the embedded deterministic subset is small enough, the conditions \eq{A} and \eq{B} continue to hold for the natural measure $\mu$ on the Cantor set. At the same time, such 
subsets can be used to construct counterexamples to restriction estimates with $p$ beyond the range guaranteed by Theorem \ref{mock-mit theorem}. Specifically, we choose the deterministic Cantor subset so that, for each $l$, the set $P_l$ of the left endpoints of its $l$-th stage intervals forms a multi-scale arithmetic progression, and let $f_l$ be essentially the indicator function of the union of these $l$-th stage intervals. The most difficult part of the proof is establishing a sufficiently large lower bound on $\|\widehat{f_l d\mu} \|_{p}$. This ultimately reduces to counting solutions to equations of the form $\sum_{i=1}^{r} a_i = \sum_{i=r+1}^{2r} a_i$, where the $a_i$ are finite-stage left endpoints of the Cantor set that lie in the support of $f_l$. The arithmetic structure of $P_l$ ensures that the number of such solutions is sufficiently large.

Our measure $\nu$ in Theorem \ref{main-thm} is a radial version of the measure $\mu$ from Theorem \ref{HL thm}. Thus the support of $\nu$ consists of nested spheres centered at the origin, where each point in the support of $\mu$ corresponds to a sphere in the support of $\nu$. At finite stages of the construction, each interval of the $l$-th iteration of the Cantor set on the line corresponds to an annulus centered at the origin, of thickness $\delta^2$ (depending on $l$). In order to construct our
counterexample, we first restrict attention to those annuli that correspond to the deterministic subset of the Cantor set. We then fix a half-line in $\RR^d$ starting at the origin, say one of the coordinate half-axes, and consider cylindrical sectors of the chosen annuli, centered on that half-line and of diameter $\delta$. This produces essentially a family of Knapp examples, stacked along the fixed half-line and parameterized by the deterministic Cantor set. Our functions $f_l$ will be the appropriately regularized indicator  functions of sets of this type. To prove our lower bound on $\|\widehat{f_l d\nu} \|_{p}$, we will use a non-trivial combination of the Hambrook-{\L}aba additive arguments and the classical Knapp calculation.

As explained above, we will rely on the construction from \cite{HL}, \cite{chen} of Cantor sets on the line that obey \eq{A} and \eq{B} with exponents $\alpha_0 = \alpha - (d-1)$ and $\beta_0 = \beta - (d-1)$, respectively, and that additionally contain lower-dimensional Cantor subsets with finite-stage endpoints forming multi-scale arithmetic progressions. Since that construction is somewhat long, we do not reproduce it here. Instead, in the next few sections we list the parameters of the construction and state the properties that will be needed in the proof of Theorem \ref{main-thm}.  The details are almost identical to those in \cite{chen}, and we encourage the interested reader to consult that paper; however, we do provide an overview of the construction in Section \ref{meaning}. The rest of the paper is devoted to the proof of Theorem \ref{main-thm}. 

\section{The numbers $\alpha_0$, $\beta_0$, $s_j$, $t_j$, $n_j$, $S_j$, $T_j$, $N_j$}\label{numbers-chap2}


We begin by defining the numerical parameters of the one-dimensional construction. 

Define $\alpha_0 = \alpha - (d-1)$ and $\beta_0 = \beta - (d-1)$. 
Let $(s_j)_{j=1}^{\infty}$, $(t_j)_{j=1}^{\infty}$, $(n_j)_{j=1}^{\infty}$ be sequences of positive integers. Define
\begin{align}
\label{sh-eq-1}
S_j = s_1 \cdots s_j, \quad T_j = t_1 \cdots t_j, \quad N_j = n_1 \cdots n_j
\end{align}
for $j \geq 1$ and $S_0 = T_0 = N_0 = 1$.
We assume the sequences satisfy the following.
\begin{flalign}
\label{sh-eq-2}
\bullet \quad &s_j \leq t_j < n_j / 2 \quad \forall j \in \NN& \\
\label{sh-eq-3}
\bullet \quad &\lim_{j \rightarrow \infty} n_j = \infty & \\
\label{sh-eq-4}
\bullet \quad &\lim_{j \rightarrow \infty} \dfrac{n_{j}}{j} = 0 & 
\\
\label{sh-eq-4.5}
\bullet \quad &\lim_{j \rightarrow \infty} \dfrac{n_{j}^{d-1}}{\ln(400 j N_{j})} = 0 &  
\\
\label{sh-eq-5} 
\bullet \quad &T_{j} \approx N_{j+1}^{\alpha_0} \ln(400 (j+1)N_{j+1}) \quad \forall j \in \NN, j \geq j_0& \\
\label{sh-eq-6}
\bullet \quad &\dfrac{S_{j+1}}{T_{j+1}} \approx N_{j+1}^{-\beta_0/2} \quad \forall j \in \NN, j \geq j_0& 
\end{flalign}
Here $j_0$ is some large positive integer. It is easy but tedious to prove that such sequences indeed exist.


\section{The sequences of sets $(A_j)_{j=0}^{\infty}$ and $(P_j)_{j=0}^{\infty}$}\label{seq-sets}

In this section, we describe the sequences of sets $(A_j)_{j=0}^{\infty}$ and $(P_j)_{j=0}^{\infty}$. These will be the left endpoints of our Cantor sets on the line.
Here and below, we will use the notation $[n]=\{0,\dots,n-1\}$ for $n \in \NN$.

Let 
\begin{align*}
P_0 = \cbr{1}, 
\quad P_{j+1} 
= \bigcup_{a \in P_j} a + \dfrac{1}{N_{j+1}}\cbr{1,3,\ldots,2s_{j+1}-1}.
\end{align*}
Thus each $P_j$ is a generalized arithmetic progression and $|P_j| = S_j$.

Let $(A_j)_{j=0}^{\infty}$ be a sequence of sets with the following properties:
\begin{flalign}
\label{seq-sets-eq-0}
\bullet \quad &\text{For each $a \in A_j$, there exists $A_{j+1,a} \subseteq \dfrac{1}{N_{j+1}}[n_{j+1}]$ with $|A_{j+1,a}| = t_{j+1}$;} \\
\label{seq-sets-eq-1}
\bullet \quad &A_0 = \cbr{1};& \\
\label{seq-sets-eq-2}
\bullet \quad &A_{j+1} = \bigcup_{a \in A_j} a + A_{j+1,a};& \\
\label{seq-sets-eq-5}
\bullet \quad &P_j \subseteq A_j;& \\
\label{seq-sets-eq-6}
\bullet \quad &\text{For each $a \in P_j$, $A_{j+1,a}$ is disjoint from $\dfrac{1}{N_{j+1}}\cbr{0,2,\ldots,2s_{j+1}}$};& \\
\label{seq-sets-eq-7}
\bullet \quad &|\widehat{\mu}(\xi)| \lesssim (1+|\xi|)^{-\beta_0/2} \text{ for all } \xi \in \RR.&
\end{flalign}
Here $\mu$ is the natural measure on the Cantor set defined through a standard iterative procedure with $A_j$ as the left endpoints of the construction intervals (see Section \ref{meas-func-chap-2}). The sequence $(A_j)_{j=0}^{\infty}$ can be constructed by making trivial modifications to the construction of Chen \cite{chen} (cf. \cite{hambrook-thesis}, \cite{HL}). 
Note that $A_j \subseteq N_j^{-1}[N_{j}]$ and $|A_j| = T_j$.


\section{The measure $\mu$}\label{meas-func-chap-2}


For $j = 0,1,2,...$, define
\begin{align*}
E_j = A_j + [0,N_{j}^{-1}] = \bigcup_{a \in A_j} [a,a+N_{j}^{-1}]
\end{align*}
and define $\mu_{j}$ to be the uniform probability measure on $E_j$, that is, 
$$
d\mu_j = \dfrac{1}{|E_j|} \one_{E_j} dx = \dfrac{N_j}{T_j} \sum_{a \in A_j} \one_{[a,a+N_{j}^{-1}]} dx.
$$
Define $\mu$ to be the weak limit of $(\mu_j)_{j=0}^{\infty}$.
The existence of the weak limit in this type of construction is standard, so we omit the proof.  
Note that $\mu$ is the so-called natural measure on the Cantor set 
\begin{equation}\label{supp-mu}
\text{supp}(\mu) = \bigcap_{j=1}^{\infty} E_j = \bigcap_{j=1}^{\infty} \bigcup_{a \in A_j} [a,a+N_{j}^{-1}] \subseteq [1,2],
\end{equation}
and
\begin{equation}\label{mu-of-interval}
\mu([a,a+N_{j}^{-1}]) = \dfrac{1}{T_j}.
\end{equation}

\section{The measure $\nu$}\label{sec-nu}

Let $\sigma$ be the uniform probability measure on the unit sphere $S^{d-1}$ in $\RR^d$. 
Define the measure $\nu$ on $\RR^d$ by
$$
d\nu(x) = |x|^{-(d-1)/2}d\mu(|x|) \otimes d\sigma(x/|x|).
$$
Clearly, $\nu$ is a finite non-trivial measure.

Let $l \in \NN$, $a \in N_{l}^{-1}[N_l]$, $e \in S^{d-1}$. Let $\delta = N_{l}^{-1/2}$.  
Define $D_{a,\delta^2}$ to be the annulus with the center at the origin, inner radius $a$, and thickness $\delta^2 = N_{l}^{-1}$. That is,
\begin{align*}
D_{a,\delta^2} = \cbr{x \in \RR^d : |x| \in  [a,a + N_{l}^{-1}] }.
\end{align*}
Then by (\ref{supp-mu}),
\begin{equation}\label{supp-nu}
\text{supp}(\nu) = \bigcap_{j=1}^{\infty} \bigcup_{a \in A_j} D_{a,\delta^2} \subseteq \cbr{x \in \RR^d : 1 \leq |x| \leq 2}.
\end{equation}
Define $C_{a,\delta^2,w,e}$ to be the sector of the annulus $D_{a,\delta^2}$ that has width $w$ and is centered on the half-line parallel to the unit vector $e$. That is, 
\begin{align*}
C_{a,\delta^2,w,e} = \cbr{x \in \RR^d : |\frac{x}{|x|} - e| \leq w/2, \, |x| \in  [a,a + N_{l}^{-1}] }.
\end{align*}
It follows easily from (\ref{mu-of-interval}) that for $w<1/2$,
\begin{align}
\label{old-cap-1}
\nu(C_{a,\delta^2,w,e}) &\approx w^{d-1} T_{l}^{-1} \quad \forall a \in A_{l}. 
\end{align}
%
%
%
%
%
%
%
%
%
%
%


\section{The meaning of the numbers, sets, and measures}\label{meaning}

In this section, we provide an overview of the construction of the finite-stage Cantor endpoint sets $(A_j)_{j=0}^{\infty}$ and $(P_j)_{j=0}^{\infty}$ and the limiting measures $\mu$ and $\nu$, and we explain how the parameters $\alpha_0$, $\beta_0$, $s_j$, $t_j$, $n_j$, $S_j$, $T_j$, $N_j$ come into play.

In general terms, the Cantor construction proceeds as follows. Start with the interval $[1,2]$. Divide it into $n_1$ equal subintervals, select $t_1$ of them, and discard the rest. For each selected subinterval, divide it into $n_2$ equal subintervals, select $t_2$ of them, and discard the rest. Continue in this way. At the $j$-th stage, we have a set $E_j$ consisting of the union of $T_j$ intervals of length $N_j^{-1}$. The left endpoints of the intervals making up $E_j$ form the set $A_j$. The support of $\mu$ is the Cantor set $E_\infty:=\cap_{j=1}^{\infty} E_j$. The support of $\nu$ is the union of those spheres in $\RR^d$ centered at the origin with radii in $E_\infty$.

Our restriction counterexample relies on a very particular choice of the subintervals in the above construction. 
We would like each $E_j$ to contain the set $P_j+[0,N_j^{-1}]$, the $j$-th stage iteration of a self-similar Cantor set whose left endpoints form a multi-scale arithmetic progression. (Thus the set $E_\infty$ contains the lower-dimensional, self-similar, strongly structured set $\cap_{j=1}^{\infty} (P_j+[0,N_j^{-1}])$, but we will not use this fact directly, and will instead work with finite iterations of the construction.) This will be essential for disrupting the restriction inequality beyond the range of exponents in Theorem \ref{mock-mit theorem}.

To this end, we need to make sure at each stage that $P_j \subseteq A_j$, but also that $A_j$ is otherwise sufficiently random for \eq{seq-sets-eq-7} to hold. 
We start with
$A_0=P_0=\{1\}$ and proceed by induction. Recall that, for $j \geq 0$, $P_{j+1} = \bigcup_{a \in P_j} a + N_{j+1}^{-1}\cbr{1,3,\ldots,2s_{j+1} - 1}$. 
Suppose that $A_j$ is given, with $P_j \subseteq A_j$. Then $E_j$ is a union of $T_j$ intervals of length $N_j^{-1}$. 
Consider one such interval, with left endpoint at some $a \in A_j$. 
We divide the interval into $n_{j+1}$ equal subintervals, with left endpoints in $a+{N_{j+1}^{-1}}[n_{j+1}]$. 
From these endpoints, we wish to select a subset $a + A_{j+1,a}$ of cardinality $t_{j+1}$; the union of the sets $a + A_{j+1,a}$ over all $a \in A_j$ will be the set $A_{j+1}$. The selection procedure depends on whether $a \in P_j$: 


\begin{itemize}
\item If $a \in P_j$, then we always start by selecting the $s_{j+1}$ endpoints that form the arithmetic progression $a + N_{j+1}^{-1}\cbr{1,3,\ldots,2s_{j+1}-1}$, thus ensuring that $P_{j+1} \subseteq A_{j+1}$. 
We then select the other $t_{j+1} - s_{j+1}$ endpoints from $a + {N_{j+1}^{-1}}\cbr{2s_{j+1}+1, \ldots, n_{j+1}}$. The selection is made in such a way that, provided \eq{sh-eq-6} holds, $\widehat{\mu}$ will decay as in \eq{seq-sets-eq-7}. A probabilistic argument is used to prove that this is possible. 
Note that endpoints in $a + N_{j+1}^{-1}\cbr{0,2,\ldots,2s_{j+1}}$ are not allowed in the selection process, in order to ensure that \eq{seq-sets-eq-6} holds. 

\medskip

\item If $a \notin P_j$, then we use a probabilistic argument to select $t_{j+1}$ endpoints from $a+{N_{j+1}^{-1}}[n_{j+1}]$ so that \eq{seq-sets-eq-7} will hold. No further modifications are needed. 
\end{itemize}

We close this section with a few words about the parameters $s_j$, $t_j$, and $n_j$ in Section \ref{numbers-chap2}. 
Essentially, we want $n_j$ and $t_j$ to be slowly growing sequences with the asymptotic dimensionality condition $t_j\sim n_j^{\alpha_0}$. 
The precise description is provided by \eq{sh-eq-2}--\eq{sh-eq-5}. These imply that the linear Cantor set $E_\infty$ has Hausdorff dimension $\alpha_0$ and that, moreover, $\mu$ satisfies $\mu((x-r,x+r)) \lesssim r^{\alpha_0}$ for all $x \in \RR, r>0$. In Section \ref{ball-nu}, this will be used to prove that $\nu$ satisfies \eq{A}.

The numbers $s_j$ denote the length of the arithmetic progressions included in the endpoint sets in the Cantor construction. For optimal counterexamples, we would like to maximize $s_j$ subject to the constraint that the set $A_j$ still be random enough for 
\eq{seq-sets-eq-7} to hold. Roughly speaking, this requires that $s_j \sim t_j n_j^{-\beta_0/2}$ asymptotically; the precise statement we need is \eq{sh-eq-6}. As we will see in Section \ref{decay-nu}, \eq{seq-sets-eq-7} implies that the measure $\nu$ satisfies \eq{B}.

\section{The ball condition for $\nu$}\label{ball-nu}

In this section, we prove $\nu$ satisfies \eq{A}.

Let $x \in \RR^d$ and $r > 0$ be given. If $2r \geq N_{1}^{-1}$, then
$$
\nu(B(x,r)) \leq \nu(\RR^d) \leq \nu(\RR^d) (2 N_1)^{\alpha} r^{\alpha}.
$$
Now suppose $0 < 2r < N_{1}^{-1}$. Choose $l \in \NN$ such that $N_{l+1}^{-1} \leq 2r \leq N_{l}^{-1}$. Assume $B(x,r)$ intersects $\supp(\nu)$ (otherwise $\nu(B(x,r)) = 0$). By (\ref{supp-nu}),
$B(x,r)$ intersects $D_{a,\delta^2}$ for some $a \in A_l$. Since $2r \leq N_{l}^{-1} = \delta^2$, there are at most two such $a \in A_l$, say $a'$ and $a''$. Moreover, if we set $e = x/|x|$, then
$$
B(x,r) \subseteq C_{a',\delta^2,\delta^2,e} \cup C_{a'',\delta^2,\delta^2,e}
$$ 
Therefore, by \eq{old-cap-1}, \eq{sh-eq-4.5}, and \eq{sh-eq-5}, we have
\begin{align*}
\nu(B(x,r)) 
&\leq \nu(C_{a',\delta^2,\delta^2,e}) + \nu(C_{a'',\delta^2,\delta^2,e}) \\
&\approx N_{l}^{-(d-1)}T_{l}^{-1}
\approx N_{l}^{-(d-1)} N_{l+1}^{-\alpha_0} (\ln(400 (l+1)N_{l+1}))^{-1} \\
&= N_{l+1}^{-(d-1) - \alpha_0} n_{l+1}^{d-1} (\ln(400 (l+1)N_{l+1}))^{-1} 
\lesssim N_{l+1}^{-(d-1) - \alpha_0} \\
&\lesssim r^{d-1 + \alpha_0} = r^{\alpha}
\end{align*}
as required.

\section{The Fourier decay of $\nu$}\label{decay-nu}

To prove that $\nu$ obeys \eq{B}, we invoke a theorem of Gatesoupe \cite{Gatesoupe}:

\begin{thm}\label{gatesoupe}
Let $\mu$ be a non-trivial measure on $\RR$ with compact support contained in $(0,\infty)$. Suppose 
$$
|\widehat{\mu}(\xi)| \lesssim \phi(|\xi|) \qquad \forall \xi \in \RR, \xi \neq 0,
$$
where $\phi: (0,\infty) \rightarrow (0,\infty)$ satisfies 
$$
\frac{1}{t} \lesssim \phi(t) \qquad \forall t \geq 1. 
$$
Then the measure
$$
d\nu(x) = |x|^{-(d-1)/2}d\mu(|x|) \otimes d\sigma(x/|x|) 
$$
satisfies 
$$
|\widehat{\nu}(\xi)| \lesssim  |\xi|^{-(d-1)/2} \phi(|\xi|) \qquad \forall \xi \in \RR^d, |\xi| \geq 1.
$$
\end{thm}

By \eq{seq-sets-eq-7}, the measure $\mu$ obeys the assumptions of the theorem with 
\mbox{$\phi(|\xi|)=(1+|\xi|)^{-\beta_0/2}$.} Recalling that $\beta=\beta_0+d-1$, we get \eq{B} as claimed.

\section{The functions $\psi_a$ and $f_l$}\label{functions}

We first fix $l \in \NN$ and define the functions $\psi_a$ for $a \in P_l$. 

For $Y \subseteq \RR^d$ and $\epsilon > 0$, let 
$$
N_{\epsilon}(Y) = \cbr{x \in \RR^d : |x-y| < \epsilon \text{ for some } y \in Y}
$$
be the $\epsilon$-neighborhood of $Y$. Let $e_d$ be the standard unit vector $(0,\ldots,0,1) \in \RR^d$.

Fix $l \in \NN$ and $\delta = N_{l}^{-1/2}$. 
For each $a \in P_l$, choose a $C^{\infty}_{c}(\RR^d)$ function $\psi_{a}$ that is equal to 1 on $C_{a,\delta^2,\delta,e_d}$, is equal to 0 outside $N_{\delta^2 / 2 }(C_{a,\delta^2,\delta,e_d})$, and satisfies $0 \leq \psi_a \leq 1$ everywhere.


Since $\psi_a \in C^{\infty}_{c}(\RR^d)$, it follows that $\widehat{\psi_a}$ is a Schwartz function. 
In Section \ref{decay-nu}, we established that $\nu$ satisfies \eq{B}. 
It follows easily 
that 
\begin{align}\label{F decay psi nu}
|\widehat{\psi_a d\nu}(\xi)| = |(\widehat{\psi_a}* \widehat{\nu})(\xi)|\lesssim_{a,l} (1+|\xi|)^{-\beta/2}.
\end{align}
The implied constant here depends on $a$ and $l$; we could remove the dependence on $a$ by choosing each $\psi_{a}$ to be a translation of a single function, but this is not important for our argument.

\begin{lemma}\label{F psi_a size lemma}
With $\psi_a$ as above, we have 
\begin{align}\label{F psi_a size}
\widehat{\psi_a d\nu}(0) \approx \delta^{d-1} T_{l}^{-1} = N_{l}^{-(d-1)/2} T_{l}^{-1}, 
\end{align}
where the implied constant depends only on $d$.
\end{lemma}

\begin{proof}
By \eq{old-cap-1}, 
\begin{align}
\label{psi-2}
\widehat{\psi_a d\nu}(0) \geq \nu(C_{a,\delta^2,\delta,e_d}) \approx \delta^{d-1} T_{l}^{-1}. 
\end{align}
On the other hand, 
\begin{align}\label{psi-1}
\widehat{\psi_a d\nu}(0) \leq \nu(N_{\delta^2 / 2 }(C_{a,\delta^2,\delta,e_d})).
\end{align}
Because of \eq{seq-sets-eq-6}, the sets $D_{a,\delta^2}$ with $a \in P_l$ are isolated from the sets $D_{a',\delta^2}$ with $a' \in A_l$, $a' \neq a$. 
Consequently, for each $a \in P_l$ we have
$$
N_{\delta^2 / 2 }(C_{a,\delta^2,\delta,e_d}) \cap N_{\delta^2 /2 }(D_{a',\delta^2}) = \emptyset \qquad \forall a' \in A_l, \, a' \neq a,
$$
so that only the annulus $D_{a,\delta^2}$ contributes to (\ref{psi-1}). 
It follows that
\begin{align*}
\widehat{\psi_a d\nu}(0) 
\leq \nu(C_{a,\delta^2,\delta+\delta^2,e_d}) 
\leq \nu(C_{a,\delta^2,2\delta,e_d})\lesssim 2^{d-1} \delta^{d-1} T_{l}^{-1},
\end{align*}
where at the last step we used \eq{old-cap-1}.
Combining this with \eq{psi-2} gives 
(\ref{F psi_a size}).
\end{proof}

Finally, we define $f_l = \sum_{a \in P_l} \psi_a$.


\section{A lower bound on $\|\widehat{f_l d\nu} \|^{2r}_{2r}$}\label{lower bound}

\begin{lemma}\label{prop-2rnorm}
Let $r \in \NN$ be such that $r \beta > d$. Then
\begin{align}\label{fldnu 2r ineq 4}
\|\widehat{f_l d\nu}\|^{2r}_{2r} \gtrsim (2r)^{-l} N_{l}^{(d+1)/2} N_{l}^{-r(d-1)} T_{l}^{-2r} S_l^{2r-1}.
\end{align}
\end{lemma}

\begin{proof}
For $x \in \RR^d$, we will write $x = (x_1,\ldots,x_d)$. Fix a small constant $\eta > 0$ to be specified later. It is important that $\eta$ does not depend on $l$. It is allowed to depend on $r$ and $d$. Recall that $\delta = N_{l}^{-1/2}$, and let $R_{\delta}$ be the box
$$
R_{\delta} = \cbr{ \xi \in \RR^d: |\xi_j| \leq \eta/\delta \hbox{ for }j=1,\dots,d-1,\  |\xi_d| \leq \eta/\delta^2 }.
$$

For every $f:\RR^d \rightarrow \CC$, we define $\tilde{f}$ by $\tilde{f}(x)=f(-x)$ for all $x \in \RR^d$.  Here and elsewhere $\lambda$ is the Lebesgue measure on $\RR^d$. Let $h:\RR^d \rightarrow [0,1]$ be a Schwartz function that is equal to 1 on $\frac{1}{4}R_{\delta}$ and has $\supp(h) \subseteq \frac{1}{2}R_{\delta}$. Define $g = h \ast \tilde{h}$. Then $g$ is a non-negative Schwartz function such that
\begin{flalign}
\label{g-eq-3}
\bullet\quad &\tilde{\widehat{g}} = |\tilde{\widehat{h}}|^2 \geq 0,& \\
\label{g-eq-4}
\bullet\quad &g \leq \lambda(\tfrac{1}{2}R_{\delta}),& \\
\label{g-eq-5}
\bullet\quad &\text{$g(\xi) \geq \lambda(\tfrac{1}{8}R_{\delta})$ for $\xi \in \tfrac{1}{8} R_{\delta}$},& \\
\label{g-eq-6}
\bullet\quad &\supp(g) \subseteq R_{\delta}.&
\end{flalign}

By \eq{g-eq-4},
\begin{align}\label{FL-1}
\|\widehat{f_l d\nu} \|^{2r}_{2r} = \int |\widehat{f_l d\nu}(\xi)|^{2r} d\xi \geq \frac{1}{\lambda(\tfrac{1}{2}R_{\delta})} \int g(\xi) |\widehat{f_l d\nu}(\xi)|^{2r} d\xi.
\end{align}
Note
$$
\widehat{f_l d\nu}(\xi) 
= \sum_{a \in P_l} e^{-2\pi i \xi \cdot a e_d} \int e^{-2\pi i \xi \cdot (x - a e_d)} \psi_a(x) d\nu(x),
$$
where $e_d$ is the standard unit vector $(0,\ldots,0,1) \in \RR^d$. Since $r \in \NN$, we have 
\begin{align*}
|\widehat{f_l d\nu}(\xi)|^{2r} &= \sum_{a_1,\ldots,a_{2r} \in P_l} e^{-2\pi i \xi \cdot  e_d(\sum_{i=1}^{r} a_i - \sum_{i=r+1}^{2r} a_i)} 
\quad \times \quad \\
&\prod_{i=1}^{r} \int e^{-2\pi i \xi \cdot (x - a_i e_d)} \psi_{a_i}(x) d\nu(x) \prod_{i=r+1}^{2r} \int e^{2\pi i \xi \cdot (x - a_i e_d)} \psi_{a_i}(x) d\nu(x).
\end{align*}
Substituting into \eq{FL-1} gives 
\begin{align}\label{FL-2}
\|\widehat{f_l d\nu} \|^{2r}_{2r} \geq \frac{1}{\lambda(\frac{1}{2}R_{\delta})} \sum_{a_1,\ldots,a_{2r} \in P_l} I(a_1,\ldots,a_{2r}),
\end{align}
where
\begin{gather*}
I(a_1,\ldots,a_{2r}) = \int g(\xi) e^{-2\pi i \xi \cdot  e_d(\sum_{i=1}^{r} a_i - \sum_{i=r+1}^{2r} a_i)} \quad \times \quad \\
\notag
\prod_{i=1}^{r} \int e^{-2\pi i \xi \cdot (x - a_i e_d)} \psi_{a_i}(x) d\nu(x) \prod_{i=r+1}^{2r} \int e^{2\pi i \xi \cdot (x - a_i e_d)} \psi_{a_i}(x) d\nu(x) d\xi
\end{gather*}
We claim that for any choice of $a_1,\ldots,a_{2r} \in P_l$, the integral $I(a_1,\ldots,a_{2r})$ is non-negative. We prove the claim in Section \ref{proof of claim}. Assume the claim for now. It then follows from \eq{FL-2} that
\begin{align}\label{FL-3}
\|\widehat{f_l d\nu} \|^{2r}_{2r} \geq \frac{1}{\lambda(\frac{1}{2}R_{\delta})} \sum_{\substack{a_1,\ldots,a_{2r} \in P_l \\ \sum_{i=1}^{r} a_i = \sum_{i=r+1}^{2r} a_i}} I(a_1,\ldots,a_{2r}).
\end{align}
Using that $\supp(g) \subseteq R_{\delta}$, we rewrite \eq{FL-3} as
\begin{align}\label{FL-3.5}
\|\widehat{f_l d\nu} \|^{2r}_{2r} \geq \frac{1}{\lambda(\frac{1}{2}R_{\delta})} \sum_{\substack{a_1,\ldots,a_{2r} \in P_l \\ \sum_{i=1}^{r} a_i = \sum_{i=r+1}^{2r} a_i}} \int_{R_{\delta}} g(\xi) P(\xi,a_1,\ldots,a_{2r}) d\xi,
\end{align}
where 
\begin{gather}\label{P product}
P(\xi,a_1,\ldots,a_{2r}) = \\
\notag
\prod_{i=1}^{r} \int e^{-2\pi i \xi \cdot (x - a_i e_d)} \psi_{a_i}(x) d\nu(x) \prod_{i=r+1}^{2r} \int e^{2\pi i \xi \cdot (x - a_i e_d)} \psi_{a_i}(x) d\nu(x).
\end{gather}

For a lower bound on $P(\cdot)$, we adapt the Knapp argument. Fix $a \in P_l$ for now, 
and let $x \in \supp(\psi_a)$ and $\xi \in R_{\delta}$. From the definition of $\psi_a$ we have
$|x_j| \leq 3 \delta$ for $j=1,\ldots,d-1$, and $|x_d - a| \leq 3 \delta^2$.
Hence
\begin{align*}
|\xi \cdot (x - a e_d)| \leq \sum_{j=1}^{d-1} |\xi_j||x_j| + |\xi_d||x_d - a| \leq 3d\eta,
\end{align*}
and, since $|e^{it} - 1| \leq |t|$ for $t \in \RR$,
\begin{align}
\label{exp - 1 bound}
|e^{\pm 2\pi i \xi \cdot (x - a e_d)} - 1| \leq 6 \pi d \eta.
\end{align}
Define $E(\xi,a)$ by 
$$
\int e^{-2\pi i \xi \cdot (x - a e_d)} \psi_a(x) d\nu(x) = \widehat{\psi_a d\nu}(0) + E(\xi,a).
$$
It follows from \eq{exp - 1 bound} that
\begin{align}
\label{E bound}
|E(\pm \xi,a)| \leq 6 \pi d \eta \widehat{\psi_{a} d\nu}(0).
\end{align}
Rewrite \eq{P product} as
\begin{align*}
P(\xi,a_1,\ldots,a_{2r}) = \prod_{i=1}^{r} \rbr{\widehat{\psi_{a_i} d\nu}(0) + E(\xi,a_i)} \prod_{i=1+r}^{2r} \rbr{\widehat{\psi_{a_i} d\nu}(0) + E(-\xi,a_i)}.
\end{align*}
Expanding this and using (\ref{E bound}), we see that
\begin{align*}
P(\xi,a_1,\ldots,a_{2r}) 
\geq \Big(1-(2^{2r} - 1) (6 \pi d \eta)\Big)  \prod_{i=1}^{2r} \widehat{\psi_{a_i} d\nu}(0)
\geq\frac{1}{2} \prod_{i=1}^{2r} \widehat{\psi_{a_i} d\nu}(0),
\end{align*}
assuming that $\eta$ is small enough depending on $r$ and $d$.

We now return to \eq{FL-3.5}. Applying our lower bound on $P(\cdot)$ and using that $g \geq 0$, we get
\begin{align}\label{FL-5}
\| \widehat{f_l d\nu} \|^{2r}_{2r} \geq \frac{1}{2 \lambda(\frac{1}{2}R_{\delta})} \sum_{\substack{a_1,\ldots,a_{2r} \in P_l \\ \sum_{i=1}^{r} a_i = \sum_{i=r+1}^{2r} a_i}} \prod_{i=1}^{2r} \widehat{\psi_{a_i} d\nu}(0) \int_{R_{\delta}} g(\xi) d\xi.
\end{align}
Then using \eq{g-eq-5}, we have
\begin{align}\label{FL-6}
\int_{R_{\delta}} g(\xi) d\xi \geq \int_{\frac{1}{8}R_{\delta}} g(\xi) d\xi \geq (\lambda(\tfrac{1}{8}R_{\delta}))^2.
\end{align}
We clearly have 
\begin{align}\label{FL-7}
\lambda(c R_{\delta}) = (c \eta)^{d} \delta^{-(d+1)} = (c \eta)^d N_{l}^{(d+1)/2}
\end{align}
for any $c > 0$. Applying \eq{FL-6} and \eq{FL-7} to \eq{FL-5} gives 
%
$$
\| \widehat{f_l d\nu} \|^{2r}_{2r} \geq  2^{-1 - 5d} \eta^d  N_{l}^{(d+1)/2} \sum_{\substack{a_1,\ldots,a_{2r} \in P_l \\ \sum_{i=1}^{r} a_i = \sum_{i=r+1}^{2r} a_i}} \prod_{i=1}^{2r} \widehat{\psi_{a_i} d\nu}(0).
$$
Now employing Lemma \ref{F psi_a size lemma} yields
\begin{align}\label{fldnu 2r ineq 3}
\| \widehat{f_l d\nu} \|^{2r}_{2r} \gtrsim N_{l}^{(d+1)/2} N_{l}^{-2r(d-1)/2} T_{l}^{-2r} M_{l,r},
\end{align}
where
$$
M_{l,r} = \abs{ \cbr{ (a_1,\ldots,a_{2r}) \in P_{l}^{2r} :  \sum_{i=1}^{r} a_i = \sum_{i=r+1}^{2r} a_i}   }. 
$$

We now work out a lower bound on $M_{l,r}$. Define
$$
P_{l}^{\oplus r} = \cbr{ \sum_{i=1}^{r} a_i : a_i \in P_l}
$$
and
$$
G(b) = \abs{ \cbr{ (a_1,\ldots,a_r) \in P_{l}^{\oplus r} : \sum_{i=1}^{r} a_i = b }}.
$$
Then
$$
M_{l,r} = \sum_{b \in P_{l}^{\oplus r}} G(b)^2.
$$
By the Cauchy-Schwarz inequality, 
\begin{align}\label{CS ineq}
\rbr{\sum_{b \in P_{l}^{\oplus r}} G(b)}^2 \leq M_{l,r} |P_{l}^{\oplus r}|.
\end{align}
To bound $M_{l,r}$, first note that
\begin{align}\label{sum g(b)}
\sum_{b \in P_{l}^{\oplus r}} G(b) = \abs{ P_l }^r = S_{l}^{r}. 
\end{align}
Next we estimate $|P_{l}^{\oplus r}|$. Each $a \in P_l$ is of the form
$$
a = 1 + \sum_{k=1}^{l} \frac{a^{(k)}}{N_k},
$$
where $a^{(k)} \in \cbr{1,3,\ldots,2s_k - 1}$ for $k=1,\ldots,l$. Therefore each $b \in P_{l}^{\oplus r}$ is of the form 
$$
b = r + \sum_{k=1}^{l} \frac{b^{(k)}}{N_k},
$$
where $b^{(k)} \in \cbr{r,r+1,\ldots,r(2s_k - 1)}$ for $k=1,\ldots,l$. Hence
\begin{align}\label{Ploplusr}
|P_{l}^{\oplus r}| \leq (2r)^l S_l.
\end{align}
Combining \eq{CS ineq}, \eq{sum g(b)}, and \eq{Ploplusr} gives
$$
M_{l,r} \geq \frac{S_{l}^{2r}}{ (2r)^l S_l} = (2r)^{-l} S_l^{2r-1}.
$$
Applying this lower bound for $M_{l,r}$ in \eq{fldnu 2r ineq 3} yields
(\ref{fldnu 2r ineq 4}).
\end{proof}

\section{The divergence of $\|\widehat{f_l d\nu} \|^{p}_{p} / \| f_l \|^{p}_{L^2(\nu)}$}\label{divergence-p}

\begin{lemma}\label{lemma-topnorm}
Let $r \in \NN$ be such that $r \beta > d$. For all $p$ with $1\leq p\leq 2r$ we have
\begin{equation}\label{e-topnorm}
\| \widehat{f_l d\nu} \|^{p}_{p} \gtrsim \frac{N_{l}^{-(p/2)(d - 1 + \beta_0) + (d+1)/2 - \alpha_0 + \beta_0/2}}{ (2r)^{l} n_{l+1}^{\alpha_0} \ln(400 (l+1)N_{l+1}) }
\end{equation}
\end{lemma}

\begin{proof}
Lemma \ref{prop-2rnorm} says
\begin{align*}
\|\widehat{f_l d\nu}\|^{2r}_{2r} \gtrsim (2r)^{-l} N_{l}^{(d+1)/2} N_{l}^{-r(d-1)} T_{l}^{-2r} S_l^{2r-1}.
\end{align*}
Applying \eq{sh-eq-5} and \eq{sh-eq-6} gives 
\begin{align}\label{num}
\|\widehat{f_l d\nu}\|^{2r}_{2r} 
&\gtrsim   \frac{N_{l}^{-r(d - 1 + \beta_0) + (d+1)/2 - \alpha_0 + \beta_0/2}}{ (2r)^{l} n_{l+1}^{\alpha_0} \ln(400 (l+1)N_{l+1}) }.
\end{align}
which is (\ref{e-topnorm}) with $p=2r$. Assume now that $1\leq p<2r$. Then
\begin{align*}
\| \widehat{f_l d\nu} \|^{2r}_{2r} 
= \int |\widehat{f_l d\nu}(\xi)|^{2r-p} |\widehat{f_l d\nu}(\xi)|^{p} d\xi 
\leq \rbr{\int f_l(x) d\nu(x)}^{2r-p}    \| \widehat{f_l d\nu} \|^{p}_{p}.
\end{align*}
By Lemma \ref{F psi_a size lemma}, the fact $|P_l| = S_l$, and \eq{sh-eq-6}, we have
\begin{align*}
\int f_l(x) d\nu(x) = \sum_{a \in P_l} \int \psi_a(x) d\nu(x) \approx N_{l}^{-(d-1)/2} T_{l}^{-1} S_{l} \approx N_l^{-(d-1 + \beta_0)/2},
\end{align*}
so that
$$
\| \widehat{f_l d\nu} \|^{p}_{p} \gtrsim N_l^{-(2r-p)(d-1 + \beta_0)/2}\| \widehat{f_l d\nu} \|^{2r}_{2r} 
$$
This together with (\ref{num}) yields (\ref{e-topnorm}).
\end{proof}

\begin{lemma}\label{lemma-bottomnorm}
For $1\leq p<\infty$, we have 
\begin{align*}
\| f_l \|^{p}_{L^2(\nu)} \approx N_{l}^{-\frac{p}{4}( d-1+ \beta_0) }.
\end{align*}
\end{lemma}

\begin{proof}
We have
\begin{align*}
\| f_l \|^{2}_{L^2(\nu)} = \int \rbr{\sum_{a \in P_l}  \psi_a(x)}^2 d\nu(x) = \sum_{a \in P_l} \int (\psi_a(x))^2 d\nu(x)
\end{align*}
because the $\psi_a$ have disjoint supports. By an argument analogous to the proof of Lemma \ref{F psi_a size lemma}, 
$$
\int (\psi_a(x))^2 d\nu(x) \approx N_{l}^{-(d-1)/2} T_{l}^{-1}.
$$
Using this, the fact $|P_l| = S_l$, and \eq{sh-eq-6} completes the proof.
\end{proof}

We now complete the proof of Theorem \ref{main-thm}. Let $1 \leq p < \infty$. Choose $r \in \NN$ such that $r \beta > d$  
and $2r\geq p$. By Lemmas \ref{lemma-topnorm} and \ref{lemma-bottomnorm}, we have
$$
\frac{\| \widehat{f_l d\nu} \|^{p}_{p}}{\| f_l \|^{p}_{L^2(\nu)}} \gtrsim \frac{N_{l}^{-(p/4)(d - 1 + \beta_0) + (d+1)/2 - \alpha_0 + \beta_0/2}}{ (2r)^{l} n_{l+1}^{\alpha_0} \ln(400 (l+1)N_{l+1}) }
$$
Because of \eq{sh-eq-3} and \eq{sh-eq-4}, the exponent of $N_l$ determines whether the right-hand side diverges. 
Specifically, we have divergence if and only if 
$-(p/4)(d - 1 + \beta_0) + (d+1)/2 - \alpha_0 + \beta_0/2 > 0.$
Recalling that $\beta_0=\beta-(d-1)$ and $\alpha_0=\alpha-(d-1)$, this translates after a little bit of algebra to
$p < (4d - 4\alpha + 2\beta)/\beta$, as required by the statement of Theorem \ref{main-thm}.

\section{Proof of the claim}\label{proof of claim}

In the proof of Lemma \ref{prop-2rnorm}, we claimed that the integral
\begin{gather*}
I(a_1,\ldots,a_{2r})
= \int g(\xi) e^{-2\pi i \xi \cdot  e_d(\sum_{i=1}^{r} a_i - \sum_{i=r+1}^{2r} a_i)} \quad \times \quad \\ 
\prod_{i=1}^{r} \int e^{-2\pi i \xi \cdot (x - a_i e_d)} \psi_{a_i}(x) d\nu(x) \prod_{i=r+1}^{2r} \int e^{2\pi i \xi \cdot (x - a_i e_d)} \psi_{a_i}(x) d\nu(x) d\xi.
\end{gather*}
is non-negative for any choice of $a_1,\ldots,a_{2r} \in P_l$. In this section, we prove this claim.

Recall the hypothesis of Lemma \ref{prop-2rnorm} is that $r \in \NN$ is such that $r \beta > d$. For $i = 1, \ldots, r$, write $m_{i}$ for the measure defined by
$$
\int f(x) dm_{i}= \int f(x-a_i e_d) \psi_{a_i}(x) d\nu(x)  \quad \forall f \in L^1(\nu).
$$
For $i = 1+r, \ldots, 2r$, write $m_i$ for the measure defined by
$$
\int f(x) dm_{i} = \int f(a_i e_d - x) \psi_{a_i}(x) d\nu(x) \quad \forall f \in L^1(\nu).
$$
By \eq{F decay psi nu}, for all $1 \leq i \leq 2r$ and $\xi \in \RR^d$, we have 
$$
|\widehat{m_i}(\xi)| 
= |e^{\pm 2\pi i \xi \cdot a_i e_d} \widehat{ \psi_{a_i} d\nu }(\pm \xi)|
\lesssim (1+|\xi|)^{-\beta/2}.
$$
The implied constant depends on $a_i$ and $l$, but this dependence is not important for our argument. 
It follows that
$$
\abs{ \mathcal{F} [\ast_{i=1}^{2r} m_i](\xi) } = \abs{ \prod_{i=1}^{2r} \mathcal{F}[m_i](\xi) } \lesssim (1+|\xi|)^{-r \beta}.  
$$
Since $r \beta > d$, it follows that $\mathcal{F} [\ast_{i=1}^{2r} m_i] \in L^1(\RR^d)$. Therefore $\ast_{i=1}^{2r} m_i$ has a continuous, non-negative, $L^1(\RR^d)$ density function $\phi$. 
Rewrite $I(a_1,\ldots,a_{2r})$ as
\begin{align*}
I(a_1,\ldots,a_{2r}) 
&= \int e^{-2\pi i \xi \cdot e_d (\sum_{i=1}^{r} a_i  - \sum_{i=r+1}^{2r} a_i )} \prod_{i=1}^{2r} \mathcal{F}[m_i](\xi) g(\xi) d\xi \\
&= \int e^{-2\pi i \xi \cdot e_d (\sum_{i=1}^{r} a_i  - \sum_{i=r+1}^{2r} a_i )} \mathcal{F}[\ast_{i=1}^{2r} m_i] (\xi) g(\xi) d\xi \\
&= \int e^{-2\pi i \xi \cdot e_d (\sum_{i=1}^{r} a_i  - \sum_{i=r+1}^{2r} a_i )} \mathcal{F}[\phi] (\xi) g(\xi) d\xi.
\end{align*}
Since $g$ is a Schwartz function, we have $g = \mathcal{F}[\tilde{\widehat{g}}]$. 
Therefore
$$
I(a_1,\ldots,a_{2r}) = \int e^{-2\pi i \xi \cdot e_d (\sum_{i=1}^{r} a_i - \sum_{i=r+1}^{2r} a_i )} \mathcal{F}[\phi \ast \tilde{\widehat{g}}] (\xi) d\xi.
$$
Since $\phi$ and $\tilde{\widehat{g}}$ are in $L^1(\RR^d)$, Young's convolution inequality implies $\phi \ast \tilde{\widehat{g}}$ is in $L^1(\RR^d)$. Moreover, since $\mathcal{F}[\phi] = \mathcal{F}[\ast_{i=1}^{2r} m_i] \in L^1(\RR^d)$ and $g$ is bounded, $\mathcal{F}[\phi \ast \tilde{\widehat{g}}] = \mathcal{F}[\phi] g$ is also in $L^1(\RR^d)$. Therefore we can apply the Fourier inversion theorem to obtain
$$
I(a_1,\ldots,a_{2r}) = (\phi \ast \tilde{\widehat{g}})\rbr{ -e_d \cdot \rbr{\sum_{i=1}^{r} a_i - \sum_{i=r+1}^{2r} a_i  }}.
$$
As $\phi$ and $\tilde{\widehat{g}}$ are both non-negative, we conclude that $I(a_1,\ldots,a_{2r})$ is non-negative, as claimed.


\begin{comment}

\affiliationone{%
   Kyle Hambrook \\
   Department of Mathematics, University of British Columbia, Vancouver, BC, V6T1Z2  \\
   Canada
   \email{hambrook@math.ubc.ca}}
\affiliationtwo{%
   Izabella {\L}aba\\
   Department of Mathematics, University of British Columbia, Vancouver, BC, V6T1Z2\\
   Canada
   \email{ilaba@math.ubc.ca}}
\affiliationthree{%
   Current address:\\
   Department of Mathematics, University of Rochester, Rochester, NY, 14627 \\
   USA
   \email{khambroo@ur.rochester.edu}}
%
\affiliationfour{~} 
%
\end{comment}

\noindent {\sc Kyle Hambrook} \newline
Department of Mathematics, University of Rochester, Rochester, NY, 14627 USA \newline
Department of Mathematics, University of British Columbia, Vancouver, BC, V6T1Z2 Canada \newline
\texttt{khambroo@ur.rochester.edu, hambrook@math.ubc.ca}

\noindent {\sc Izabella {\L}aba} \newline
Department of Mathematics, University of British Columbia, Vancouver, BC, V6T1Z2 Canada \newline
\texttt{ilaba@math.ubc.ca}
\end{document}